\documentclass[11pt]{amsart} \usepackage{latexsym, amssymb}
\usepackage[T1]{fontenc}

\newtheorem{thm}{Theorem}[section]
\newtheorem{lemma}[thm]{Lemma}
\newtheorem{prop}[thm]{Proposition}
\newtheorem{cor}[thm]{Corollary}

\theoremstyle{definition}

\newtheorem{nrmk}[thm]{Remark}

\theoremstyle{remark}
\newtheorem*{rmk}{Remark}

\newcommand{\Z}{\mathbb{Z}}

\newcommand{\curly}[1]{\mathcal{#1}}

\newcommand{\n}{\mathbb{N}}

\renewcommand{\S}{\curly{S}}

\newcommand{\la}{\curly{L}}

\renewcommand{\to}{\rightarrow}

\def \f{\mathbb F}

\def \<{\langle}
\def \>{\rangle}

\def \*Z {{{^*}\Z}}
\def \((  {(\!(}
\def \)) {)\!)}

\numberwithin{equation}{section}

\def \Th{\operatorname{Th}}
\def \R{\mathcal R}
\def \u{\mathcal U}

\makeatletter

\def\indsym#1#2{%
  \setbox0=\hbox{$\m@th#1x$}%
  \kern\wd0%
  \hbox to 0pt{\hss$\m@th#1\mid$\hbox to 0pt{$\m@th#1^{#2}$}\hss}%
  \lower.9\ht0\hbox to 0pt{\hss$\m@th#1\smile$\hss}%
  \kern\wd0}
\newcommand{\ind}[1][]{\mathop{\mathpalette\indsym{#1}}}

\def\nindsym#1#2{%
  \setbox0=\hbox{$\m@th#1x$}%
  \kern\wd0%
  \hbox to 0pt{\hss$\m@th#1\not$\kern1.4\wd0\hss}
  \hbox to 0pt{\hss$\m@th#1\mid$\hbox to 0pt{$\m@th#1^{\,#2}$}\hss}%
  \lower.9\ht0\hbox to 0pt{\hss$\m@th#1\smile$\hss}%
  \kern\wd0}

\allowdisplaybreaks[2]

\DeclareMathOperator{\Aut}{Aut}
\DeclareMathOperator{\id}{id}
\DeclareMathOperator{\SL}{SL}

\title[Model Companion of Tracial Von Neumann Algebras]{The theory of tracial von Neumann algebras does not have a model companion}
\author{Isaac Goldbring, Bradd Hart, Thomas Sinclair}
\thanks{Goldbring's work was partially supported by NSF grant DMS-1007144.}

\address {Department of Mathematics, Statistics, and Computer Science, University of Illinois at Chicago, Science and Engineering Offices M/C 249, 851 S. Morgan St., Chicago, IL, 60607-7045}
\email{isaac@math.uic.edu}
\urladdr{http://www.math.uic.edu/~isaac}

\address{Department of Mathematics and Statistics, McMaster University, 1280 Main Street W., Hamilton, Ontario, Canada L8S 4K1}
\email{hartb@mcmaster.ca}
\urladdr{http://www.math.mcmaster.ca/~bradd}

\address{Department of Mathematics, University of California, Los Angeles, 520 Portola Plaza, Box 951555, Los Angeles, CA, 90095-1555}
\email{thomas.sinclair@math.ucla.edu}
\urladdr{http://www.math.ucla.edu/~thomas.sinclair}

\begin{document}
\begin{abstract}
In this note, we show that the theory of tracial von Neumann algebras does not have a model companion.  This will follow from the fact that the theory of any locally universal, McDuff II$_1$ factor does not have quantifier elimination.  We also show how a positive solution to the Connes Embedding Problem implies that there can be no model-complete theory of II$_1$ factors.
\end{abstract}
\maketitle

\section{Introduction}

The model theoretic study of operator algebras is at a relatively young stage in its development (although many interesting results have already been proven, see \cite{FHS1},\cite{FHS2}, \cite{FHS3}) and thus there are many foundational questions that need to be answered.  In this note, we study the question that appears in the title:  does the theory of tracial von Neumann algebras have a model companion?  (Recall that a theory is said to be \emph{model-complete} if every embedding between models of the theory is elementary and a model-complete theory $T'$ is a \emph{model companion} of a theory $T$ if every model of $T$ embeds into a model of $T'$ and vice-versa.)  We show that the answer to this question is:  no!  Indeed, we prove that a locally universal, McDuff II$_1$ factor cannot have quantifier elimination.  (See below for the definitions of \emph{locally universal} and \emph{McDuff}.)  Since a model companion of the theory of tracial von Neumann algebras will have to be a model completion as well as the theory of a locally universal, McDuff II$_1$ factor, the result follows.  

We then pose a weaker question:  can there exist a model-complete theory of II$_1$ factors?  Here, we show that a positive solution to the \emph{Connes Embedding Problem} implies that the answer is once again:  no!

Another motivation for this work came from considering independence relations in II$_1$ factors.  Although all II$_1$ factors are unstable (see \cite{FHS1}), it is still possible that there are other reasonably well-behaved independence relations to consider.  Indeed, the independence relation stemming from conditional expectation is a natural candidate.  In the end of this note, we show how the failure of quantifier elimination seems to pose serious hurdles in showing that conditional expectation yields a strict independence relation in the sense of \cite{Adler}. 

We thank Dima Shlyakhtenko for patiently explaining Brown's work when we posed the question to him of the existence of non-extendable embeddings of pairs $\mathcal M\subset \mathcal N$ into $\R^\omega$.  (See the proof of Theorem \ref{noqeR} below.)

Throughout, $\la$ denotes the signature for tracial von Neumann algebras and $\R$ denotes the hyperfinite II$_1$ factor.  We recall that $\R$ embeds into any II$_1$ factor.  We will say that a von Neumann algebra is \emph{$\R^\omega$-embeddable} if it embeds into $\R^\u$ for some $\u\in \beta\n\setminus \n$.  If $M$ is $\R^\omega$ embeddable, then $M$ embeds into $\R^\u$ for \emph{all} $\u\in \beta\n\setminus \n$; see Corollary 4.15 of \cite{FHS2}.  For this reason, we fix $\u\in \beta\n\setminus \n$ throughout this note.






\section{Model Companions}

In the proof of our first theorem, we use the crossed product construction for von Neumann algebras; a good reference is \cite[Chapter 4]{BO}.

\begin{thm}\label{noqeR}
$\Th(\R)$ does not have quantifier elimination.
\end{thm}

\begin{proof}

It is enough to find separable, $\R^\omega$-embeddable tracial von Neumann algebras $M\subset N$ and an embedding $\pi:M\to \R^\u$ that does not extend to an embedding $N\to \R^\u$.  Indeed, if this is so, let $N_1$ be a separable model of $\Th(\R)$ containing $N$.  Then $\pi$ does not extend to an embedding $N_1\to \R^\u$; since $\R^\u$ is $\aleph_1$-saturated, this shows that $\Th(\R)$ does not have QE.

In order to achieve the goal of the above paragraph, we claim that it is enough to find a countable discrete group $\Gamma$ such that $L(\Gamma)$ is $\R^\omega$-embeddable, an embedding $\pi:L(\Gamma)\to \R^\u$, and $\alpha\in \Aut(L(\Gamma))$ such that there exists no unitary $u\in \R^\u$ satisfying $(\pi\circ \alpha)(x) = u\pi(x)u^*$ for all $x\in L(\Gamma)$. (We should remark that we are using the usual trace on $L(\Gamma)$ and that $\Aut(L(\Gamma))$ refers to the group of $*$-automorphisms preserving this trace.)
First, we abuse notation and also use $\alpha$ to denote the homomorphism $\mathbb Z\to \Aut(L(\Gamma))$ which sends the generator of $\mathbb Z$ to the aforementioned $\alpha$.  Set $\mathcal M = L(\Gamma)$ and $\mathcal N = \mathcal M\rtimes_\alpha \mathbb Z$.  Then $N$ is a tracial von Neumann algebra.  Moreover, we have that $\mathcal N$ is $\R^\omega$-embeddable if and only if $\mathcal M$ is---in fact, this is true for any crossed product algebra $\mathcal M\rtimes_\alpha G$ where $G$ is amenable \cite[Prop. 3.4(2)]{A}.  Now suppose, towards a contradiction, that $\pi$ were to extend to an embedding $\widetilde\pi: \mathcal N\to \R^\u$. If $u\in L(\mathbb Z)\subset \mathcal M\rtimes_\alpha \mathbb Z$ is the generator of $\mathbb Z$, then setting $\tilde u = \widetilde\pi(u)\in \R^\u$, we would have that $\tilde u\pi(x)\tilde u^* = \pi(uxu^*) = \pi(\alpha(x))$ for all $x\in \mathcal M$, contradicting the fact that $\pi\circ\alpha$ is not unitarily conjugate to the embedding $\pi$ in $\R^\u$.

An explicit construction of $\Gamma$, $\pi$ and $\alpha$ as above has already appeared in the work of N.\ P.\ Brown \cite{Brown}. Indeed, by Corollary 6.11 of \cite{Brown}, we may choose $\Gamma = \SL(3,\mathbb Z) \ast \mathbb Z$ and $\alpha = \id\ast\theta$ for any nontrivial $\theta\in \Aut(L(\mathbb Z))$.

\end{proof}

We say that a separable II$_1$ factor $\S$ is \emph{locally universal} if every separable II$_1$ factor embeds into $\S^\u$.  (By \cite[Corollary 4.15]{FHS2}, this notion is independent of $\u$.)  In \cite{FHS3}, it is shown that a locally universal II$_1$ factor exists.  The \emph{Connes Embedding Problems} (CEP) asks whether $\R$ is locally universal.

We say that a separable II$_1$ factor $M$ is \emph{McDuff} if $M\otimes \R\cong M$.  For example, $\R$ is McDuff as is $M\otimes \R$ for any separable II$_1$ factor $M$. By examining Brown's argument in \cite{Brown}, we see that the only properties of $\R$ that are used (other than it being finite) is that $L(\Gamma)$ (for $\Gamma$ as in the previous proof) is $\R^\omega$-embeddable and that $\R$ is McDuff.  We thus have:

\begin{cor}\label{noqe}
If $\S$ is a locally universal, McDuff II$_1$ factor, then $\Th(\S)$ does not have QE.
\end{cor}

Let $T_0$ be the theory of tracial von Neumann algebras in the signature $\la$.  $T_0$ is a universal theory; see \cite{FHS2}.  Let $T$ be the theory of II$_1$ factors, a $\forall \exists$-theory by \cite{FHS2}.  Moreover, since every tracial von Neumann algebra is contained in a II$_1$ factor, we see that $T_0=T_\forall$.  Thus, an existentially closed model of $T_0$ is a model of $T$.

By \cite[Proposition 3.9]{FHS3}, there is a set $\Sigma$ of $\forall \exists$-sentences in the language of tracial von Neumann algebras such that $M$ is McDuff if and only if $M\models \Sigma$.  Since every II$_1$ factor is contained in a McDuff II$_1$ factor (as $M\subseteq M\otimes \R)$, it follows that an existentially closed II$_1$ factor is McDuff.

We can now prove our main result:

\begin{thm}\label{nomodcomp}
$T_0$ does not have a model companion.
\end{thm}

\begin{proof}
Suppose that $T$ is a model companion for $T_0$.  Since $T_0$ is univerally axiomatizable and has the amalgamation property (see \cite[Chapter 4]{BO}), we have that $T$ has QE.

Fix a separable model $\S$ of $T$.  Then $\S$ is a locally universal II$_1$ factor.  Indeed, given an arbitrary separable II$_1$ factor $M$, we have a separable model $\S_1\models T$ containing $M$.  Since $\S^\u$ is $\aleph_1$-saturated, we have that $\S_1$ embeds into $\S^\u$, yielding an embedding of $M$ into $\S^\u$.  Meanwhile, since $T$ is the theory of existentially closed models of $T_0$, we see that $\S$ is McDuff.  Thus, by Corollary \ref{noqe}, $T$ does not have QE, a contradiction.

\end{proof}

\section{Model Complete II$_1$ Factors}

While we have proven that the theory of tracial von Neumann algebras does not have a model companion, at this point it is still possible that there is a model complete theory of II$_1$ factors.  In this section, we show that a positive solution to the CEP implies that there is no model-complete theory of II$_1$ factors.

We begin by observing the following:

\begin{lemma}\label{elem}
Every embedding $\R\to \R^\omega$ is elementary.
\end{lemma}

\begin{proof}
This follows from the fact that every embedding $\R\to \R^\omega$ is unitarily equivalent to the diagonal embedding; see \cite{Jung}.
\end{proof}

\begin{rmk}
The previous lemma shows that $\R$ is the unique prime model of its theory.  Indeed, to show that $\R$ is a prime model of its theory, by Downward L\"owenheim-Skolem (DLS), it is enough to show that whenever $M\equiv \R$ is separable, then $\R$ elementarily embeds into $M$.  Well, since $\R^\u$ is $\aleph_1$-saturated, we have that $M$ elementarily embeds into $\R^\u$.  Composing an embedding $\R\to M$ with the elementary embedding $M\to \R^\u$ and applying Lemma \ref{elem}, we see that the embedding $\R\to M$ is elementary.  
\end{rmk}

\begin{prop}\label{R}
Suppose that $M$ is an $\R^\omega$-embeddable II$_1$ factor such that $\Th(M)$ is model-complete.  Then $M\equiv \R$.
\end{prop}

\begin{proof}
Without loss of generality, we may assume that $M$ is separable.  Fix embeddings $\R\to M$ and $M\to \R^\u$.  By Lemma \ref{elem}, the composition $$\R\to M\to \R^\u$$ is elementary.  By DLS, we can take a separable elementary substructure $\R_1$ of $\R^\u$ such that $M$ embeds in $\R_1$; observe that the composition $\R\to M\to \R_1$ is elementary.  By DLS again, take a separable elementary substructure $M_1$ of $M^\u$ such that $\R_1$ embeds in $M_1$.  We now repeat this process with $M_1$:  embed $M_1$ in $\R^{\u}$, take separable elementary substructure $\R_2$ of $\R^{\u}$ such that $M_1$ embeds in $\R_2$ and then embed $\R_2$ in a separable elementary substructure $M_2$ of $M^{\u}$.  Iterate this construction countably many times, obtaining
$$\R\to M\to \R_1\to M_1\to \R_2\to M_2\to \cdots,$$ where each $\R_n$ is a separable elementary substructure of $\R^\u$ and each $M_i$ is a separable elementary substructure of $M^\u$. 
Set $\R_\omega=\bigcup_n \R_n=\bigcup_n M_n$.  Then $\R$ is an elementary substructure of $\R_\omega$ since  $\R\to \R_1$ is elementary and $\R_n\to \R_{n+1}$ is elementary for each $n\geq 1$.
Meanwhile, observe that $M_n\equiv M$ for each $n$, so by model-completeness of $\Th(M)$, we have that the $M_n$'s form an elementary chain, whence $M$ is an elementary substructure of $\R_\omega$.  Consequently, $\R\equiv M$.  
\end{proof}

\begin{nrmk}
Proposition \ref{R} provides immediate examples of non-model complete theories of II$_1$ factors.  Indeed, for $m\geq 2$, the von Neumann group algebra of the free group on $m$ generators, $L(\f_m)$, is $\R^\omega$-embeddable but not elementarily equivalent to $\R$ (see 3.2.2 in \cite{FHS3}), whence $\Th(L(\f_m))$ is not model-complete.  It is an outstanding problem in operator algebras whether or not $L(\f_m)\cong L(\f_n)$ for all $m,n\geq 2$.  A weaker, but still seemingly difficult, question is whether or not $L(\f_m)\equiv L(\f_n)$ for all $m,n\geq 2$.  (An equivalent formulation of this question is whether or not there is $\u\in \beta\n\setminus \n$ such that $L(\f_m)^\u\cong L(\f_n)^\u)$?)  Suppose this latter question has an affirmative answer.  Then we see that the theory of free group von Neumann algebras is not model-complete, mirroring the corresponding fact that the theory of free groups is not model-complete.  However, the natural embeddings $\f_m\to \f_n$, for $m< n$, are elementary.  \emph{Assuming $L(\f_m)\equiv L(\f_n)$, are the natural embeddings $L(\f_m)\to L(\f_n)$, for $m<n$, elementary?}
\end{nrmk}

\begin{cor}
Assume that the CEP has a positive solution.  Then there is no model-complete theory of II$_1$ factors.
\end{cor}

\begin{proof}
Suppose that $T$ is a model-complete theory of II$_1$ factors.  By the positive solution to the CEP and Proposition \ref{R}, $T=\Th(\R)$.  Meanwhile, a positive solution to the CEP implies that $T_\forall=T_0$, whence $T$ is a model companion for $T_0$, contradicting Theorem \ref{nomodcomp}.
\end{proof}

%




\section{Concluding Remarks}

Theorem \ref{noqeR} presents a major hurdle in trying to understand the model theory of II$_1$ factors.  In particular, it places a major roadblock in trying to understand potential independence relations in theories of II$_1$ factors.  Indeed, although any II$_1$ factor is unstable (see \cite{FHS1}), one might wonder whether the natural notion of independence stemming from noncommutative probability theory might show that some II$_1$ factor is (real) rosy (see \cite{Adler} for the definition of rosy theory).  More precisely, fix some ``large'' II$_1$ factor $M$ and consider the relation $\ind$ on ``small'' subsets of $M$ given by $A\ind_C B$ if and only if, for all $a\in \langle AC\rangle$, $E_{\langle C\rangle}(a)=E_{\langle BC\rangle}(a)$.  Here, $\langle * \rangle$ denotes the von Neumann subalgebra generated by $*$ and $E_{\langle *\rangle}$ is the conditional expectation (or orthogonal projection) map $E_{\langle *\rangle}:L^2M\to L^2\langle *\rangle $.  In trying to verify some of the natural axioms for an independence relation (see \cite{Adler}), one runs into trouble when trying to verify the extension axiom:  If $B\subseteq C\subseteq D$ and $A\ind_B C$, can we find $A'$ realizing the same type as $A$ over $C$ such that $A'\ind_B D$?  If $M=\R^\u$ and ``small'' means ``countable,'' then it seems quite likely that one could find an $A'$ with the same \emph{quantifier-free type} as $A$ over $C$ that is independent from $D$ over $B$ as quantifier-free types are determined by moments.  Without quantifier-elimination, it seems quite difficult to prove the extension property for this purported notion of independence.  (The question of whether or not the independence relation arising from conditional expectation yields a strict independence relation was also discussed in \cite{Ben}.)

\end{document}